\date{26 September 2008}

\documentclass[11pt,reqno]{amsart}
\usepackage{latexsym,amsmath,amsfonts,amscd,amssymb}
\usepackage{graphics}

\theoremstyle{plain}
\newtheorem{theorem}{Theorem}[section]
\newtheorem{corollary}[theorem]{Corollary}
\newtheorem{lemma}[theorem]{Lemma}
\newtheorem{proposition}[theorem]{Proposition}
\theoremstyle{definition}

\numberwithin{equation}{section}

\newcommand{\RR}{\mathbb{R}}
\newcommand{\CC}{\mathbb{C}}

\newcommand{\PP}{\mathbb{P}}

\newcommand{\ZZ}{\mathbb{Z}}

\newcommand{\x}{\times}

\newcommand{\la}{\langle}
\newcommand{\ra}{\rangle}

\newcommand{\Id}{\mathrm{Id}}

\DeclareMathOperator{\Hom}{Hom}

\DeclareMathOperator{\tr}{tr}

\newcommand{\SL}{\mathrm{SL}(2,{\mathbb{C}})}

\title{The $\SL$-character varieties of torus knots}


\author{Vicente Mu\~noz}
  \address{Instituto de Ciencias Matem{\'a}ticas CSIC-UAM-UC3M-UCM \\
  Consejo Superior de Investigaciones Cient{\'\i}ficas \\ Serrano 113 bis
  \\ 28006 Madrid \\ Spain}
  \address{Facultad de Matem\'{a}ticas \\ Universidad Complutense
  de Madrid \\ Plaza Ciencias 3
  \\ 28040 Madrid \\ Spain}
  \email{vicente.munoz@imaff.cfmac.csic.es}

\thanks{Partially supported through grant MEC
(Spain) MTM2007-63582}

\begin{document}
\maketitle

\begin{abstract}
  Let $G$ be the fundamental group of the complement of the
  torus knot of type $(m,n)$. This has a presentation $G=\la x,y \, |
  \, x^m=y^n\ra$. We find the geometric description of the character
  variety $X(G)$ of characters of representations of $G$ into $\SL$.
\end{abstract}

\bigskip

\bigskip

Since the foundational work of Culler and Shalen \cite{CS}, 
the varieties of $\SL$-characters have been extensively studied. 
Given a manifold $M$, the variety of representations of
$\pi_1(M)$ into $\SL$ and the variety of characters of such representations both
contain information of the topology of $M$. This is specially
interesting for $3$-dimensional manifolds, where the fundamental
group and the geometrical properties of the manifold are 
strongly related.

This can be used to study knots $K\subset S^3$, by analysing the
$\SL$-character variety of the fundamental group of the knot complement
$S^3-K$. In this paper, we study the case of the torus knots $K_{m,n}$ 
of any type $(m,n)$. The case $(m,n)=(m,2)$ was analysed in \cite{Oller} and
the general case was recently determined in \cite{Martin-Oller} by 
a method different from ours.

\bigskip

\noindent \textbf{Acknowledgements.} \
The author wishes to thank the referee for useful comments, specially 
for pointing out the reference \cite{Martin-Oller}.

\section{Character varieties} \label{sec:introduction}

A \textit{representation} of a group $G$ in $\SL$ is a
homomorphism $\rho: G\to \SL$.
Consider a finitely presented group $G=\la x_1,\ldots, x_k | r_1,\ldots, r_s \ra$,
and let $\rho: G\to \SL$ be a representation. Then $\rho$ is completely
determined by the $k$-tuple $(A_1,\ldots, A_k)=(\rho(x_1),\ldots, \rho(x_k))$
subject to the relations $r_j(A_1,\ldots, A_k)=0$, $1\leq j \leq s$. Using
the natural embedding $\SL\subset \CC^4$, we can identify the space of
representations as
 \begin{eqnarray*}
 R(G) &=& \Hom(G, \SL) \\
  &=& \{(A_1,\ldots, A_k) \in \SL^k \, | \,
 r_j(A_1,\ldots, A_k)=0, \,  1\leq j \leq s \}\subset \CC^{4k}\, .
 \end{eqnarray*}
Therefore $R(G)$ is an affine algebraic set.

We say that two representations $\rho$ and $\rho'$ are
equivalent if there exists $P\in \SL$ such that $\rho'(g)=P^{-1} \rho(g) P$,
for every $g\in G$.
This produces an action of $\SL$ in $R(G)$. The moduli space of representations
is the GIT quotient
 $$
 M(G) = \Hom(G, \SL) // \SL \, .
 $$

A representation $\rho$ is \textit{reducible} if the elements of
$\rho(G)$ all share a common eigenvector, otherwise $\rho$ is
\textit{irreducible}.

Given a representation $\rho: G\to \SL$, we define its
\textit{character} as the map $\chi_\rho: G\to \CC$,
$\chi_\rho(g)=\tr \rho (g)$. Note that two equivalent
representations $\rho$ and $\rho'$ have the same character, and
the converse is also true if $\rho$ or $\rho'$ is irreducible
\cite[Prop.\ 1.5.2]{CS}.

There is a character map $\chi: R(G)\to \CC^G$, $\rho\mapsto
\chi_\rho$, whose image
 $$
 X(G)=\chi(R(G))
 $$
is called the \textit{character variety of $G$}. Let us give
$X(G)$ the structure of an algebraic variety. By the results of
\cite{CS}, there exists a collection $g_1,\ldots, g_a$ of elements of
$G$ such that $\chi_\rho$ is determined by $\chi_\rho(g_1),\ldots,
\chi_\rho(g_a)$, for any $\rho$. Such collection gives a map
 $$
  \Psi:R(G)\to \CC^a\, , \qquad
  \Psi(\rho)=(\chi_\rho(g_1),\ldots, \chi_\rho(g_a))\, .
 $$
We have a bijection $X(G)\cong \Psi(R(G))$. This endows $X(G)$ with
the structure of an algebraic variety. Moreover, this is
independent of the chosen collection as proved in \cite{CS}.

\begin{lemma}\label{lem:MyX}
  The natural algebraic map $M(G)\to X(G)$ is a bijection.
\end{lemma}

\begin{proof}
  The map $R(G)\to X(G)$ is algebraic and $\SL$-invariant, hence
  it descends to an algebraic map $\varphi:M(G)\to X(G)$. Let us
  see that $\varphi$ is a bijection.

  For $\rho$ an irreducible
  representation, if $\varphi(\rho)=\varphi(\rho')$ then $\rho$
  and $\rho'$ are equivalent representations; so they represent
  the same point in $M(G)$.

  Now suppose that $\rho$ is reducible. Consider $e_1\in \CC^2$ the
  common eigenvector of all $\rho(g)$. This gives a sub-representation
  $\rho':G\to \CC^*$ of $G$. We have a quotient representation
  $\rho''=\rho /\rho' :G\to \CC^*$, defined as the representation induced by
  $\rho$ in the quotient space $\CC^2/ \la e_1\ra$. As characters,
  $\rho''=\rho'{}^{-1}$.  The
  representation $\rho'\oplus \rho''$ is the \textit{semisimplification} of
  $\rho$. It is in the closure of the $\SL$-orbit through $\rho$.
  Clearly, $\chi_\rho(g)=\rho'(g)+\rho'(g)^{-1}$.
  Now if $\rho$ and $\tilde\rho$ are two reducible representations
  and $\varphi(\rho)=\varphi(\tilde\rho)$, then their
  semisimplifications have the same character, that is
   $$
   \chi_\rho(g)= \chi_{\tilde\rho}(g)\Rightarrow \rho'(g)+\rho'(g)^{-1}
   =\tilde\rho'(g)+\tilde\rho'(g)^{-1} \,.
   $$
  Therefore $\rho'=\tilde\rho'$ or $\rho'=\tilde\rho'{}^{-1}$. In
  either case $\rho$ and $\tilde\rho$ represent the same point
  in $M(G)$, which is actually the point represented by $\rho'\oplus
  \rho'^{-1}$.
\end{proof}

\section{Character varieties of torus knots} \label{sec:torus}

Let $T^2=S^1 \times S^1$ be the $2$-torus and consider the standard embedding
$T^2\subset S^3$. Let $m,n$ be a pair of coprime positive integers. Identifying
$T^2$ with the quotient $\RR^2/\ZZ^2$, the image of the straight line $y=\frac{m}{n}
x$ in $T^2$ defines the \textit{torus knot} of type $(m,n)$, which we shall denote
as $K_{m,n}\subset S^3$ (see \cite[Chapter 3]{Rolfsen}).

For any knot $K\subset S^3$, we denote by $G(K)$ the fundamental group of the exterior
$S^3-K$ of the knot. It is known that
 $$
  G_{m,n}= G(K_{m,n}) \cong \la x,y \, | \, x^m= y^n \,\ra \,.
 $$
The
purpose of this paper is to describe the character variety
$X(G_{m,n})$.

In \cite{Oller}, the character variety $X(G_{m,2})$ is computed.
We want to extend the result to arbitrary $m,n$, and give an
argument simpler than that of \cite{Oller}.

After the completion of this work, we became aware of the paper \cite{Martin-Oller}
where the character varieties of $X(G_{m,n})$ are determined (even without
the assumption of $m,n$ being coprime). However, our method is more direct
than the one presented in \cite{Martin-Oller}.

\bigskip

To start with, note that
 \begin{equation}\label{eqn:R}
 R(G_{m,n})= \{ (A,B)\in \SL \, | \, A^m= B^n\}\,.
 \end{equation}
Therefore we shall identify a representation $\rho$ with a pair of
matrices $(A,B)$ satisfying the required relation $A^m=B^n$.

We decompose the character variety
 $$
 X(G_{m,n})=X_{red} \cup X_{irr}\, ,
 $$
where $X_{red}$ is the subset consisting of the characters of
reducible representations (which is a closed subset by \cite{CS}),
and $X_{irr}$ is the closure of the subset consisting of the
characters of irreducible representations.

\begin{proposition}\label{prop:red}
 There is an isomorphism $X_{red}\cong \CC$. The correspondence is defined
 by
  $$
  \rho=\left(
  A=\left( \begin{array}{cc} t^n & 0 \\ 0 & t^{-n} \end{array}\right),
  B=\left( \begin{array}{cc} t^m & 0 \\ 0 & t^{-m} \end{array}\right)
  \right) \mapsto s=t+t^{-1} \in \CC \, .
  $$
\end{proposition}

\begin{proof}
  By the discussion in Lemma \ref{lem:MyX}, an element in
  $X_{red}$ is described as the character of a split
  representations $\rho=\rho'\oplus \rho'^{-1}$. This means that
  in a suitable basis,
  \begin{equation}\label{eqn:red-pair}
  A=\left( \begin{array}{cc} \lambda & 0 \\ 0 & \lambda^{-1}
  \end{array}\right)\, , \,
  B=\left( \begin{array}{cc} \mu & 0 \\ 0 & \mu^{-1}
  \end{array}\right)\, .
  \end{equation}
  The equality $A^m=B^n$ implies $\lambda^m=\mu^n$. Therefore
  there is a unique $t\in \CC$ with $t\neq 0$ such that
   $$
     \left\{ \begin{array}{l} \lambda =t^n , \\ \mu=t^m . \end{array}\right.
   $$
  (Here we use the coprimality of $(m,n)$.) Note that
  the pair $(A,B)$ is well-defined up to permuting the two vectors in
  the basis. This corresponds to the change $(\lambda,\mu)\mapsto
 (\lambda^{-1},\mu^{-1})$, which in turn corresponds to $t\mapsto
 t^{-1}$. So $(A,B)$ is parametrized by $s=t+t^{-1}\in \CC$.
\end{proof}

\begin{lemma}\label{lem:red}
  Suppose that $\rho=(A,B)\in R(G_{m,n})$. In any of the following
  cases:
   \begin{enumerate}
   \item[(a)] $A^m=B^n\neq \pm \Id$,
   \item[(b)] $A=\pm \Id$ or $B=\pm \Id$,
   \item[(c)] $A$ or $B$ is non-diagonalizable,
   \end{enumerate}
  the representation $\rho$ is reducible.
\end{lemma}

\begin{proof}
First suppose that $A$ is diagonalizable with eigenvalues
$\lambda,\lambda^{-1}$, and suppose that $\lambda^m \neq \pm 1$.
Then there is a basis $e_1,e_2$ in which $A=\left(
\begin{array}{cc} \lambda & 0 \\ 0 & \lambda^{-1}
\end{array}\right)$, which is well-determined up to multiplication
of the basis vectors by non-zero scalars. Then
  $$
  B^n= A^m=  \left( \begin{array}{cc} \lambda^m & 0 \\ 0 & \lambda^{-m} \end{array}\right)
 $$
is a diagonal matrix, different from $\pm \Id$. Therefore $B$ must
be diagonal in the same basis, $B=\left( \begin{array}{cc} \mu & 0 \\
0 & \mu^{-1} \end{array}\right)$, with $\lambda^m=\mu^n$. This
proves the reducibility in case (a).

Now suppose that $A=\lambda\Id$, $\lambda= \pm 1$. Then $B^n=
\lambda^m \Id$, so it must be that $B$ is diagonalizable. Using
a basis in which $B$ is diagonal, we get the reducibility in
case (b).

Finally, suppose that $A$ is not diagonalizable. Then there is a
suitable basis on which $A$ takes the form $A=\left(
\begin{array}{cc} \lambda & 1 \\ 0 & \lambda
\end{array}\right)$, with $\lambda=\pm 1$. Clearly
  $$
  B^n= A^m= \lambda^m  \left( \begin{array}{cc} 1 & m\lambda \\ 0 & 1 \end{array}\right)
  $$
and so
 $$
 B= \left( \begin{array}{cc} \mu & x \\ 0 & \mu \end{array}\right)
 $$
with $\mu =\pm 1$, $\mu^n=\lambda^m$ and $\mu n x= \lambda m$. In
this basis, the vector $e_1$ is an eigenvector for both $A$ and
$B$. Hence the representation $(A,B)$ is reducible, completing the
case (c).
\end{proof}

\begin{proposition}\label{prop:irr}
Let $X_{irr}^o$ be the set of irreducible characters, and $X_{irr}$ its closure. Then
  \begin{eqnarray*}
  X_{irr}^o  &=&
 \{ (\lambda, \mu,r) \, | \, \lambda^m=\mu^n= \pm 1, \lambda\neq \pm 1,
 \mu \neq \pm 1, r\in \CC-\{0,1\} \} / \ZZ_2\x\ZZ_2\, , \\
  X_{irr} &=&
 \{ (\lambda, \mu,r) \, | \, \lambda^m=\mu^n= \pm 1, \lambda\neq \pm 1,
 \mu \neq \pm 1, r\in \CC \} / \ZZ_2\x\ZZ_2\, .
  \end{eqnarray*}
where $\ZZ_2\x\ZZ_2$ acts as $(\lambda,\mu,r)\sim (\lambda^{-1}, \mu, 1-r) \sim (\lambda,
\mu^{-1},1-r)\sim (\lambda^{-1},\mu^{-1},r)$.
\end{proposition}

\begin{proof}
Let $\rho=(A,B)$ be an element of $R(G_{m,n})$ which is
an irreducible representation. By Lemma \ref{lem:red}, $A$ is
diagonalizable but not equal to $\pm \Id$, and $A^m=\pm \Id$. So the
eigenvalues $\lambda, \lambda^{-1}$ of $A$ satisfy $\lambda^m=\pm 1$ and
$\lambda\neq \pm 1$. Analogously, $B$ is diagonalizable but not
equal to $\pm \Id$, with eigenvalues $\mu, \mu^{-1}$, with
$\mu^n=\pm 1$, $\mu \neq \pm 1$. Moreover,
  $$
  \lambda^m=\mu^n \, .
  $$

We may choose a basis $\{e_1,e_2\}$ under which $A$ diagonalizes.
This is well-defined up to multiplication of $e_1$ and $e_2$ by
two non-zero scalars.
Let $\{f_1,f_2\}$ be a basis under which $B$ diagonalizes, which
is well-defined up to multiplication of $f_1$, $f_2$ by non-zero
scalars. Then $\{[e_1],[e_2],[f_1],[f_2]\}$ are four points of
the projective line $\PP^1=\PP(\CC^2)$. Note that the pair
$(A,B)$ is irreducible if and only if the four points are different.

The only invariant of four points in $\PP^1$ is the double ratio
 \begin{equation}\label{eqn:double-ration}
 r=([e_1]:[e_2]:[f_1]:[f_2]) \in \PP^1-\{0,1,\infty\}= \CC-\{0,1\} \, .
 \end{equation}
So $(A,B)$ is parametrized, up to the action of $\SL$, by $(\lambda,\mu,r)$.
Permuting the two basis vectors $e_1,e_2$ corresponds to
$(\lambda,\mu,r) \mapsto (\lambda^{-1},\mu, 1-r)$, since
 $$
 ([e_2]:[e_1]:[f_1]:[f_2]) =1-([e_1]:[e_2]:[f_1]:[f_2]).
 $$
Analogously, permuting the two basis vectors $f_1,f_2$ corresponds
to $(\lambda,\mu,r) \mapsto (\lambda, \mu^{-1}, 1-r)$. Note that this
gives an action of $\ZZ_2\x \ZZ_2$ and $X_{irr}^o$ is the quotient
of the set of $(\lambda,\mu,r)$ as above by this action.

To describe the closure of $X_{irr}^o$, we have to allow $f_1$ to
coincide with $e_1$. This corresponds to $r=1$ (the same happens if
$f_2$ coincides with $e_2$). In this case, $e_1$ is an
eigenvector of both $A$ and $B$, so the representation $(A,B)$
has the same character as its semisimplification $(A',B')$ given by
  $$
  A'=\left(\begin{array}{cc} \lambda & 0 \\ 0 & \lambda^{-1} \end{array}\right),
  B'=\left(\begin{array}{cc} \mu & 0 \\ 0 & \mu^{-1} \end{array}\right) .
 $$
This means that the point $(\lambda,\mu,1)$ corresponds under the identification
$X_{red}\cong\CC$ given by Proposition \ref{prop:red} to $s_1=t_1+t_1^{-1}$, where
$t_1\in \CC$ satisfies
 \begin{equation}\label{eqn:A}
 \left\{\begin{array}{l} \lambda=t_1^n , \\ \mu=t_1^m . \end{array}\right.
 \end{equation}
Also, we have to allow $f_1$ to coincide with $e_2$ (or $f_2$ to coincide with $e_1$).
This corresponds to $r=0$. The representation $(A,B)$ has semisimplification $(A',B')$
where
 $$
  A'=\left(\begin{array}{cc} \lambda & 0 \\ 0 & \lambda^{-1} \end{array}\right),
  B'=\left(\begin{array}{cc} \mu^{-1} & 0 \\ 0 & \mu \end{array}\right).
  $$
So the point $(\lambda,\mu,1)$ corresponds to $s_0=t_0+t_0^{-1}\in X_{red}\cong \CC$, where
$t_0\in \CC$ satisfies
 \begin{equation}\label{eqn:B}
 \left\{\begin{array}{l} \lambda=t_0^n , \\ \mu^{-1}=t_0^m . \end{array}\right.
 \end{equation}
\end{proof}

Proposition \ref{prop:irr} says that $X_{irr}$ is a collection of
$\frac{(m-1)(n-1)}2$ lines. A pair $(\lambda,\mu)$ with $\lambda^m=\pm 1$ and
$\mu^n=\pm 1$ is given as
 \begin{equation}\label{eqn:lm}
 \lambda=e^{\pi i k/m},  \ \mu=e^{\pi i k'/n},
 \end{equation}
where $0\leq k<2m$, $0\leq k'< 2n$. The condition $\lambda\neq\pm 1$, $\mu\neq
\pm 1$ gives $k\neq 0,m$, $k'\neq 0,n$. Finally, the $\ZZ_2\times \ZZ_2$-action
allows us to restrict to $0<k<m$, $0<k'<n$. The condition $\lambda^m=\mu^n$ means that
 $$
 k\equiv k' \pmod 2.
 $$

Denote by $X_{irr}^{k,k'}$ the line of $X_{irr}$ corresponding to the values of $k,k'$.
Then
 $$
 X_{irr}=\bigsqcup_{0<k<m,0<k'<n \atop k\equiv k' \pmod2} X_{irr}^{k,k'}\, .
 $$
The line $X_{irr}^{k,k'}$ intersects $X_{red}$ in two points. This gives a collection
of $(m-1)(n-1)$ points in $X_{red}$, which are defined as follows: under the
identification $X_{red}\cong \CC$, these are the points $s_l=t_l+t_l^{-1}$, where
 $$
 t_l=e^{\pi i l/nm},
 $$
and $0<l<mn$, $m \hspace{-1.5mm}\not| \, l$, $n\hspace{-1.5mm} \not| \, l$.
Assume that $n$ is odd (note that either $m$ or $n$ should be odd). Then
from (\ref{eqn:A}) and (\ref{eqn:B}),
the line $X_{irr}^{k,k'}$ intersects at the points $s_{l_0}, s_{l_1}\in X_{red}$
where
 \begin{eqnarray*}
 &  & n l_0 \equiv k \pmod m ,  \qquad m l_0 \equiv n-k' \pmod n \, ,\\
 &  & n l_1 \equiv k \pmod m ,  \qquad m l_1 \equiv k' \pmod n \, .
\end{eqnarray*}
These two points are different since $k'\not\equiv n-k' \pmod n$,
as $n$ is odd.

The following is a picture of $X(G_{m,n})$.

\begin{figure}[h]
\centering
\resizebox{7cm}{!}{\includegraphics{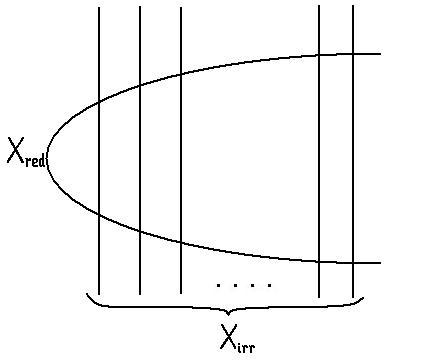}}    
\end{figure}
\centerline{\sc Figure 1}

In the case $(m,n)=(m,2)$, this result coincides with \cite[Corollary 4.2]{Oller}.

\section{The algebraic structure of $X(G_{m,n})$}\label{sec:alg}

We want to give a geometric realization of $X(G_{m,n})$ which
shows that the algebraic structure of this variety is that of a
collection of rational lines as in Figure 1 intersecting with nodal
curve singularities.

The map $R(G_{m,n})\to \CC^3$, $\rho=(A,B)\mapsto (\tr(A),\tr(B),\tr(AB))$,
defines a map
 $$
 \Psi:X(G_{m,n})\to \CC^3\, .
 $$

\begin{theorem} \label{thm:main}
 The map $\Psi$ is an isomorphism with its image $C=\Psi(X(G_{m,n}))$.
 $C$ is a curve consisting of $\frac{(n-1)(m-1)}2 +1$ irreducible
 components, all of them smooth and isomorphic to $\CC$. They intersect
 with nodal normal crossing singularities following the pattern in Figure 1.
\end{theorem}

\begin{proof}
 Let us look first at $\Psi_0=\Psi|_{X_{red}}: X_{red}\to\CC^3$. For a
given $\rho=(A,B)\in X_{red}$, with the shape given in Proposition \ref{prop:red},
 we have that
 $$
 \Psi_0:s=t+t^{-1}\mapsto   (t^n+t^{-n},t^m+t^{-m},t^{n+m}+t^{-(n+m)})\, .
 $$
This map is clearly injective: the image recovers $\{t^n,t^{-n}\},
\{t^m,t^{-m}\}, \{t^{n+m},t^{-(n+m)}\}$.
>From this, we recover $\{(t^n,t^m),(t^{-n},t^{-m})\}$ and hence
the pair $t,t^{-1}$ (since $n,m$ are coprime).

Let us see that $\Psi_0$ is an immersion. The differential is
  \begin{equation} \label{eqn:dPsi/dt}
  \frac{d\Psi_0}{dt}= \left( n t^{-n-1}(t^{2n}-1) , m t^{-m-1}(t^{2m}-1) ,
  (n+m) t^{-n-m-1}(t^{2n+2m}-1) \right) .
  \end{equation}
This is non-zero at all $t\neq \pm 1$. As $\frac{ds}{dt}\neq 0$, we
have $\frac{d\Psi_0}{ds}\neq (0,0,0)$.
For $t=\pm 1$, we note that $\frac{ds}{dt}= t^{-2}(t^{2}-1)$, so
  $$
  \frac{d\Psi_0}{ds}= \left( n t^{-n+1}\frac{t^{2n}-1}{t^2-1} , m t^{-m+1}\frac{t^{2m}-1}{t^2-1} , (n+m)
  t^{-n-m+1}\frac{t^{2n+2m}-1}{t^2-1}\right) ,
  $$
which is non-zero again.

Now, consider a component of $X_{irr}$ corresponding to a pair $(\lambda,\mu)$.
Take $r\in \CC$. Fix the basis $\{e_1,e_2\}$ of $\CC^2$ which is given as the
eigenbasis of $A$. Let $\{f_1,f_2\}$ be the eigenbasis of $B$. As
the double ratio $(0:\infty:1:r/(r-1))=r$, we can take $f_1=(1,1)$ and
$f_2=(r-1,r)$. This corresponds to the matrices:
 \begin{eqnarray*}
  A &=&\left(\begin{array}{cc} \lambda & 0 \\ 0 & \lambda^{-1} \end{array}\right), \\
  B &=&\left(\hspace{-1mm}\begin{array}{cc} 1 & r-1 \\ 1 & r \end{array}\hspace{-1mm}\right)
  \left(\hspace{-1mm}\begin{array}{cc} \mu & 0 \\ 0 & \mu^{-1} \end{array}\hspace{-2mm}\right)
  \left(\hspace{-1mm}\begin{array}{cc} 1 & r-1 \\ 1 & r \end{array}\hspace{-1mm}\right)^{-1} =
  \left(\hspace{-1mm}\begin{array}{cc} r(\mu-\mu^{-1})+\mu^{-1} & (1-r)(\mu-\mu^{-1})
  \\ r(\mu-\mu^{-1}) & \mu-r(\mu-\mu^{-1}) \end{array}\hspace{-1mm}\right) \, .
  \end{eqnarray*}
Therefore:
 $$
 \Psi(A,B)=(\tr(A),\tr(B),\tr(AB))=
 (\lambda+\lambda^{-1}, \mu^{-1}+\mu, (\lambda \mu^{-1}+\lambda^{-1} \mu) +
  r (\lambda -\lambda^{-1}) (\mu-\mu^{-1})).
  $$
The image of this component is a line in $\CC^3$. Its direction vector is $(0,0,1)$.
At an intersection point with $\Psi_0(X_{red})$, the tangent vector to
$\Psi_0(X_{red})$, given in (\ref{eqn:dPsi/dt}), has non-zero first and second component,
since $\lambda=t^{n}$, $\mu=t^{m}$ and $t\neq 0$, $\lambda^2\neq 1$, $\mu^2\neq 1$.
So the intersection of these components is a transverse nodal singularity.

  Finally, note that the map $\Psi:X(G_{m,n})\to C$ is an algebraic
  map, it is a bijection, and $C$ is a nodal curve (the mildest possible type
  of singularities). Therefore $\Psi$ must be an isomorphism.
\end{proof}

\begin{corollary}
  $M(G)\cong X(G)$, for $G=G_{m,n}$.
\end{corollary}

\begin{proof}
  By Lemma \ref{lem:MyX}, $\varphi:M(G)\to X(G)$ is an algebraic map which is a bijection.
  As the singularities of $X(G)$ are just transverse nodes, $\varphi$ must be an
  isomorphism.
\end{proof}

\end{document}